\documentclass[a4paper,11pt,titlepage,english]{smfart}
\usepackage{axodraw}
\usepackage{amsfonts}
\usepackage{amsmath}
\usepackage{amssymb}
\usepackage{color}
\usepackage{graphicx}
\usepackage{xspace}
\input xy
\xyoption{all}
 \font \eightrm=cmr8

 \newcommand{\nc}{\newcommand}

\newtheorem{thm}{Theorem}
\newtheorem{exam}{Example}

\newtheorem{lem}[thm]{Lemma}
\newtheorem{prop}[thm]{Proposition}
\newtheorem{defn}{Definition}
\newtheorem{rmk}[thm]{Remark}

\def\diagramme #1{\vskip 4mm \centerline {#1} \vskip 4mm}

\nc{\BA}{{\Bbb A}} \nc{\CC}{{\Bbb C}} \nc{\DD}{{\Bbb D}}
\nc{\EE}{{\Bbb E}} \nc{\FF}{{\Bbb F}} \nc{\GG}{{\Bbb G}}
\nc{\HH}{{\Bbb H}} \nc{\LL}{{\Bbb L}} \nc{\NN}{{\Bbb N}}
\nc{\PP}{{\Bbb P}} \nc{\QQ}{{\Bbb Q}} \nc{\RR}{{\Bbb R}}
\nc{\TT}{{\Bbb T}} \nc{\VV}{{\Bbb V}} \nc{\ZZ}{{\Bbb Z}}
\nc{\Cal}[1]{{\mathcal {#1}}}
\nc{\mop}[1]{\mathop{\hbox {\rm #1} }}
\nc{\smop}[1]{\mathop{\hbox {\eightrm #1} }}
\nc{\mopl}[1]{\mathop{\hbox {\rm #1} }\limits}
\nc{\frakg}{{\frak g}}
\nc{\g}[1]{{\frak {#1}}}
\def \restr#1{\mathstrut_{\textstyle |}\raise-8pt\hbox{$\scriptstyle #1$}}
\def \srestr#1{\mathstrut_{\scriptstyle |}\hbox to
  -1.5pt{}\raise-4pt\hbox{$\scriptscriptstyle #1$}}
\nc{\wt}{\widetilde}
\nc{\wh}{\widehat}
\nc{\un}{\hbox{\bf 1}}
\nc{\redtext}[1]{\textcolor{red}{\tt #1}}
\nc{\bluetext}[1]{\textcolor{blue}{#1}}
\nc{\comment}[1]{[[{\tt {#1}}]] }
\nc{\R}{{\mathbb R}}

\nc\fleche[1]{\mathop{\hbox to #1 mm{\rightarrowfill}}\limits}

\def\semi{\mathrel{\raise 1.2pt\hbox{${\scriptscriptstyle |}$}\joinrel\kern -3.8pt\mathrel{\times}}}

\def\racine{\,{\scalebox{0.07}{
\begin{picture}(29,29) (360,-285)
    \SetWidth{6}
    \SetColor{Black}
    \Vertex(375,-271){20}
  \end{picture}
  }}\,}
  \def\racineun{\,{\scalebox{0.5}{
  \begin{picture}(43,27) (342,-266)
    \SetWidth{1.0}
    \SetColor{Black}
    \GOval(352,-256)(9,9)(0){0.882}
    \Text(350,-260)[lb]{\Large{\Black{$1$}}}
  \end{picture}
  }}\,}
  \def\racinedeux{\,{\scalebox{0.5}{
  \begin{picture}(43,27) (342,-266)
    \SetWidth{1.0}
    \SetColor{Black}
    \GOval(352,-256)(9,9)(0){0.882}
    \Text(350,-260)[lb]{\Large{\Black{$2$}}}
  \end{picture}
  }}\,}
  \def\racinetrois{\,{\scalebox{0.5}{
  \begin{picture}(43,27) (342,-266)
    \SetWidth{1.0}
    \SetColor{Black}
    \GOval(352,-256)(9,9)(0){0.882}
    \Text(350,-260)[lb]{\Large{\Black{$3$}}}
  \end{picture}
  }}\,}
  \def\racinen{\,{\scalebox{0.5}{
  \begin{picture}(43,27) (342,-266)
    \SetWidth{1.0}
    \SetColor{Black}
    \GOval(352,-256)(9,9)(0){0.882}
    \Text(350,-260)[lb]{\Large{\Black{$n$}}}
  \end{picture}
  }}\,}
  \def\arbaf{\,{\scalebox{0.5}{
 \begin{picture}(106,125) (318,-165)
    \SetWidth{1.0}
    \SetColor{Black}
    \GOval(354,-155)(9,9)(0){0.882}
    \GOval(353,-117)(9,9)(0){0.882}
    \Line(353,-126)(353,-145)
    \GOval(351,-85)(9,9)(0){0.882}
    \Line(352,-93)(352,-109)
    \GOval(375,-62)(9,9)(0){0.882}
    \GOval(328,-61)(9,9)(0){0.882}
    \Line(333,-69)(344,-78)
    \Line(368,-68)(355,-80)
    \Text(369,-161)[lb]{\Large{\Black{$(a,1)$}}}
    \Text(368,-122)[lb]{\Large{\Black{$(e,1)$}}}
    \Text(365,-89)[lb]{\Large{\Black{$(h,2)$}}}
    \Text(389,-63)[lb]{\Large{\Black{$(d,1)$}}}
    \Text(336,-61)[lb]{\Large{\Black{$(c,2)$}}}
  \end{picture}
  }}\,}

 \def\arbae{\,{\scalebox{0.5}{
 \begin{picture}(89,130) (340,-161)
    \SetWidth{1.0}
    \SetColor{Black}
    \GOval(352,-151)(9,9)(0){0.882}
    \GOval(351,-114)(9,9)(0){0.882}
    \Line(352,-123)(352,-142)
    \GOval(350,-79)(9,9)(0){0.882}
    \GOval(350,-46)(9,9)(0){0.882}
    \GOval(381,-100)(9,9)(0){0.882}
    \Line(351,-88)(352,-103)
    \Line(373,-105)(359,-110)
    \Line(351,-54)(351,-70)
    \Text(365,-158)[lb]{\Large{\Black{$(a,1)$}}}
    \Text(362,-128)[lb]{\Large{\Black{$(e,1)$}}}
    \Text(394,-105)[lb]{\Large{\Black{$(d,1)$}}}
    \Text(361,-85)[lb]{\Large{\Black{$(h,2)$}}}
    \Text(363,-52)[lb]{\Large{\Black{$(c,2)$}}}
  \end{picture}
  }}\,} 
 \def\arbad{\,{\scalebox{0.5}{
  \begin{picture}(89,130) (340,-161)
    \SetWidth{1.0}
    \SetColor{Black}
    \GOval(352,-151)(9,9)(0){0.882}
    \GOval(351,-114)(9,9)(0){0.882}
    \Line(352,-123)(352,-142)
    \GOval(350,-79)(9,9)(0){0.882}
    \GOval(350,-46)(9,9)(0){0.882}
    \GOval(381,-100)(9,9)(0){0.882}
    \Line(351,-88)(352,-103)
    \Line(373,-105)(359,-110)
    \Line(351,-54)(351,-70)
    \Text(365,-158)[lb]{\Large{\Black{$(a,1)$}}}
    \Text(362,-128)[lb]{\Large{\Black{$(e,1)$}}}
    \Text(394,-105)[lb]{\Large{\Black{$(c,2)$}}}
    \Text(361,-85)[lb]{\Large{\Black{$(h,2)$}}}
    \Text(363,-52)[lb]{\Large{\Black{$(d,1)$}}}
  \end{picture}
    }}\,} 
 \def\arbac{\,{\scalebox{0.5}{
 \begin{picture}(142,104) (294,-187)
    \SetWidth{1.0}
    \SetColor{Black}
    \GOval(352,-177)(9,9)(0){0.882}
    \GOval(351,-140)(9,9)(0){0.882}
    \Line(352,-149)(352,-168)
    \GOval(349,-98)(9,9)(0){0.882}
    \GOval(304,-121)(9,9)(0){0.882}
    \GOval(389,-116)(9,9)(0){0.882}
    \Line(313,-124)(342,-136)
    \Line(349,-106)(350,-132)
    \Line(382,-122)(357,-135)
    \Text(367,-182)[lb]{\Large{\Black{$(a,1)$}}}
    \Text(364,-147)[lb]{\Large{\Black{$(e,1)$}}}
    \Text(401,-124)[lb]{\Large{\Black{$(d,1)$}}}
    \Text(360,-104)[lb]{\Large{\Black{$(h,2)$}}}
    \Text(312,-122)[lb]{\Large{\Black{$(c,2)$}}}
  \end{picture}
  }}\,} 
 \def\arbab{\,{\scalebox{0.5}{
 \begin{picture}(59,57) (341,-234)
    \SetWidth{1.0}
    \SetColor{Black}
    \GOval(352,-224)(9,9)(0){0.882}
    \GOval(351,-187)(9,9)(0){0.882}
    \Line(352,-196)(352,-215)
    \Text(365,-231)[lb]{\Large{\Black{$(e,1)$}}}
    \Text(362,-200)[lb]{\Large{\Black{$(h,2)$}}}
  \end{picture}
 }}\,}
 \def\arbxl{\,{\scalebox{0.5}{
 \begin{picture}(59,57) (341,-234)
    \SetWidth{1.0}
    \SetColor{Black}
    \GOval(352,-224)(9,9)(0){0.882}
    \Text(365,-231)[lb]{\Large{\Black{$(x,l)$}}}
  \end{picture}
 }}\,}
 \def\arbabxy{\,{\scalebox{0.5}{
 \begin{picture}(59,57) (341,-234)
    \SetWidth{1.0}
    \SetColor{Black}
    \GOval(352,-224)(9,9)(0){0.882}
    \GOval(351,-187)(9,9)(0){0.882}
    \Line(352,-196)(352,-215)
    \Text(365,-231)[lb]{\Large{\Black{$(x,k)$}}}
    \Text(362,-200)[lb]{\Large{\Black{$(y,l)$}}}
  \end{picture}
 }}\,}
 \def\arbabyx{\,{\scalebox{0.5}{
 \begin{picture}(59,57) (341,-234)
    \SetWidth{1.0}
    \SetColor{Black}
    \GOval(352,-224)(9,9)(0){0.882}
    \GOval(351,-187)(9,9)(0){0.882}
    \Line(352,-196)(352,-215)
    \Text(365,-231)[lb]{\Large{\Black{$(y,l)$}}}
    \Text(362,-200)[lb]{\Large{\Black{$(x,k)$}}}
  \end{picture}
 }}\,}
 \def\arbabst{\,{\scalebox{0.5}{
 \begin{picture}(59,57) (341,-234)
    \SetWidth{1.0}
    \SetColor{Black}
    \GOval(352,-224)(9,9)(0){0.882}
    \GOval(351,-187)(9,9)(0){0.882}
    \Line(352,-196)(352,-215)
    \Text(365,-231)[lb]{\Large{\Black{$(v,\left|T\right|)$}}}
    \Text(362,-200)[lb]{\Large{\Black{$(w,\left|S\right|)$}}}
  \end{picture}
 }}\,} 
 \def\arbaa{\,{\scalebox{0.5}{
 \begin{picture}(106,88) (314,-203)
    \SetWidth{1.0}
    \SetColor{Black}
    \GOval(352,-193)(9,9)(0){0.882}
    \GOval(351,-156)(9,9)(0){0.882}
    \Line(352,-165)(352,-184)
    \GOval(324,-134)(9,9)(0){0.882}
    \GOval(375,-132)(9,9)(0){0.882}
    \Line(357,-149)(369,-138)
    \Line(333,-141)(345,-150)
    \Text(365,-199)[lb]{\Large{\Black{$(a,1)$}}}
    \Text(363,-164)[lb]{\Large{\Black{$(b,3)$}}}
    \Text(385,-141)[lb]{\Large{\Black{$(d,1)$}}}
    \Text(334,-136)[lb]{\Large{\Black{$(c,2)$}}}
  \end{picture}
    }}\,} 
\def\echelunun{\,{\scalebox{0.5}{
\begin{picture}(45,63) (341,-230)
    \SetWidth{1.0}
    \SetColor{Black}
    \GOval(352,-220)(9,9)(0){0.882}
    \GOval(351,-183)(9,9)(0){0.882}
    \Line(352,-191)(352,-210)
    \Text(350,-224)[lb]{\Large{\Black{$1$}}}
    \Text(349,-187)[lb]{\Large{\Black{$1$}}}
  \end{picture}
}}\,}
\def\echelundeux{\,{\scalebox{0.5}{
\begin{picture}(45,63) (341,-230)
    \SetWidth{1.0}
    \SetColor{Black}
    \GOval(352,-220)(9,9)(0){0.882}
    \GOval(351,-183)(9,9)(0){0.882}
    \Line(352,-191)(352,-210)
    \Text(350,-224)[lb]{\Large{\Black{$1$}}}
    \Text(349,-187)[lb]{\Large{\Black{$2$}}}
  \end{picture}
}}\,}
\def\echeldeuxun{\,{\scalebox{0.5}{
\begin{picture}(45,63) (341,-230)
    \SetWidth{1.0}
    \SetColor{Black}
    \GOval(352,-220)(9,9)(0){0.882}
    \GOval(351,-183)(9,9)(0){0.882}
    \Line(352,-191)(352,-210)
    \Text(350,-224)[lb]{\Large{\Black{$2$}}}
    \Text(349,-187)[lb]{\Large{\Black{$1$}}}
  \end{picture}
}}\,}
\def\couronne{\,{\scalebox{0.5}{
\begin{picture}(91,58) (317,-235)
    \SetWidth{1.0}
    \SetColor{Black}
    \GOval(352,-225)(9,9)(0){0.882}
    \Text(350,-229)[lb]{\Large{\Black{$1$}}}
    \GOval(327,-193)(9,9)(0){0.882}
    \GOval(374,-194)(9,9)(0){0.882}
    \Line(367,-201)(357,-217)
    \Line(333,-198)(346,-217)
    \Text(326,-199)[lb]{\Large{\Black{$1$}}}
    \Text(373,-198)[lb]{\Large{\Black{$1$}}}
  \end{picture}
}}\,}
\def\echelununun{\,{\scalebox{0.5}{
\begin{picture}(47,99) (339,-194)
    \SetWidth{1.0}
    \SetColor{Black}
    \GOval(352,-184)(9,9)(0){0.882}
    \GOval(351,-147)(9,9)(0){0.882}
    \Line(352,-155)(352,-174)
    \Text(350,-188)[lb]{\Large{\Black{$1$}}}
    \Text(349,-151)[lb]{\Large{\Black{$1$}}}
    \GOval(349,-112)(9,9)(0){0.882}
    \Line(350,-121)(351,-137)
    \Text(347,-116)[lb]{\Large{\Black{$1$}}}
  \end{picture}
}}\,}

\begin{document}
\title{ The pre-Lie operad as a deformation of NAP }

\author{ Abdellatif Sa\" idi}
\address{ Facult\'e des Sciences de Monastir, Avenue de l'Environnement 5019 Monastir, Tunisie}
\address{Universit\'e Blaise Pascal, Laboratoire de math\'ematiques UMR 6620 et CNRS, Aubi\`ere, France}
         \email{Abdellatif.Saidi@fsm.rnu.tn}
         \email{abdellatif@math.univ-bpclermont.fr}

\date{November 9th. 2010}
\begin{abstract}
 We will define a family of multigraded operads $\mathcal{O}^{\lambda}$  dependent on a scalar parameter $\lambda$  which is a deformation of the operad $\mathcal{O}^0$. Forgetting the multi-graduation gives back the pre-Lie operad for $\lambda=1$ and the NAP operad for $\lambda=0$.
\end{abstract}
\maketitle
{\bf{MS classification (2010):}} 05C05, 16W30,18D50.


\tableofcontents

\section{Introduction}

 In this paper, $K$ is a field of characteristic zero. Operads are the correct framework to model many sorts of algebras, such as associative, commutative, Lie, Leibniz, dendriform and many others. The present paper will mainly focus on the two operads pre-Lie and NAP, governing pre-Lie algebras and non-associative permutative algebras respectively.\\
 
 The notion of operad (the terminology  "Operad" is due to J. P. May) appeared in the seventies in algebraic topology (J. Stasheff, J. P. May, J. M. Boardman, R. M. Vogt) \cite{Renaiss}. There has been since the nineties a renewed interest in this theory due to the discovery of the relationships with graph cohomology, Koszul duality, representation theory, combinatorics and quantum field theory \cite{Renaiss}.\\
 
This paper is organized as follows: the two first sections are  devoted to introductory background on operads, with an emphasis in the second section on the species point of view \cite{AJ} \cite{Chap2}. The third section is reminder of the pre-Lie and NAP operads introduced by F. Chapoton and M. Livernet in \cite{ChaLiv} and \cite{Liv}. The fourth section is devoted to the introduction of the notion of multigraded operads. Roughly speaking, these objects are closely related to operads colored with positive integers, except that we also take the additive  semi-group structure of $\mathbb{N}^*$ into account. The prototype is $Endop(V)$ where $V=\bigoplus_{n\geq 1} V_{n}$ is a positively graded vector space.\\

In the fifth section we return to rooted trees in the framework of multigraded operads, by giving a weight $\left|v\right|\in\mathbb{N}^*$ for any vertex $v$. We show how to consider pre-Lie as a deformation of NAP by introducing a family $\mathcal{O}^{\lambda}$ of multigraded operads, where $\lambda$ is any scalar in the field $K$. Setting $\lambda=0$ gives back the NAP operad and $\lambda=1$ gives back the pre-Lie operad when the multigraduation is forgotten. Finally in the last section we study the free $\mathcal{O}^{\lambda}$-algebra with one generator and we show that $\mathcal{O}^{\lambda}$ is binary and quadratic, by obtaining it as a quotient of a free binary multigraded operad by an ideal of deformed NAP relations.\\

\noindent
{\bf Acknowledgements:} I greatly thank Fr\'ed\'eric Chapoton for very helpful discussions. I also thank my two advisors Dominique Manchon and Mohamed Selmi for constant help and support.
 \section{Preliminaries and definitions}
 An operad is a combinatorial device coined for coding ''types of algebras''. Hence, for example, a Lie algebra is an algebra over some operad denoted by LIE and the operad As governs associative algebras. An operad $\mathcal{P}$ \cite{L}, \cite{GLE} (in the symmetric monoidal category of $K$ vector space) is given by a collection of vector spaces $\mathcal{P}(n)_{n\geq0}$, a right action of the symmetric group $S_n$ on $\mathcal{P}(n)$, and a collection of compositions:
 $$\begin{array}{ccccl}
 \circ_i&:&\mathcal{P}(n)\otimes\mathcal{P}(m)&\longrightarrow&\mathcal{P}(m+n-1),~~~~~i=1,...,n\\
& & (a,b)&\longmapsto& a\circ_i b
 \end{array}$$
 which $a\in\mathcal{P}(n)$ and $b\in\mathcal{P}(m)$  satisfies the following axioms:
 \begin{itemize}
 \item The two associativity conditions:
 \begin{eqnarray}
 (a\circ_i b)\circ_{i+j-1}c&=&a\circ_i (b\circ_j c),~~~~~~~~~~ \forall~i\in\{1,...,n\},~j\in\{1,...,m\},\\
 (a\circ_i b)\circ_{m+j-1}c&=&(a\circ_j c)\circ_{i}b,~~~~~~~~\forall~i,j\in\{1,...,n\},~ i< j
 \end{eqnarray}
 \item The unit axiom: there exists an object $\un\in\mathcal{P}(1)$ where:
 \begin{eqnarray}
 \un\circ a&=&a\\
 a\circ_i \un&=&a,~~~~~~\forall ~~i\in\{1,...,n\},~~a\in\mathcal{P}(n).
 \end{eqnarray}
 \item The equivariance axiom: for any $\sigma\in S_n ,\tau\in S_m$, we have:
 \begin{equation}
 a.\sigma \circ_{\sigma(i)}b.\tau=(a\circ_i b).\rho(\sigma,\tau),
 \end{equation}
 where $\rho(\sigma,\tau)\in S_{n+m-1}$ is defined by letting  $\tau$ permute the set $E_i = \{i,i+1,...,i+m-1\}$ of cardinality $m$, and then by letting $\sigma$ permute the set $\{1,...,i-1, E_i ,i+m,...,m+n-1\}$ of cardinality $n$. 
 \end{itemize}
The global composition is defined by $\gamma$:
 $$\begin{array}{ccccl}
 \gamma&:&\mathcal{P}(n)\otimes\mathcal{P}(k_1 )\otimes...\otimes\mathcal{P}(k_n )&\longrightarrow&\mathcal{P}(k_1 +...+k_n )\\
  & &(a;b_1 ,...,b_n )&\longmapsto & (\ldots((a\circ_n b_n )\circ_{n-1}b_{n-1})\ldots)\circ_1 b_1 .
 \end{array}$$
 The partial composition $\circ_i$ is defined by :
 \begin{equation}
 a\circ_i b=\gamma(a;\un,\ldots,\un,b,\un\ldots, \un)
 \end{equation}
 The operad $\mathcal{P}$ is augmented if $\mathcal{P}(0)=\{0\}$ and $\mathcal{P}(1)=K.\un$.
\begin{exam}{\bf{ The operad $Endop(V)$ where $V$ is a $K$-vector space}}\\
We work in the category of vector spaces over a field $K$, endowed with the usual tensor product. let $V$ be a  vector space. For any $n\geq1$, we define  $V^{\otimes n}$ by induction: $V^{\otimes 1}=V$ and $V^{\otimes n+1}=V\otimes V^{\otimes n}$. Let $Endop(V)(n)=\mathcal{L}(V^{\otimes n},V)$ the space of linear maps from  $V^{\otimes n}$ to $V$. The right action of the symmetric group $S_n$ on $Endop(V)(n)$ is defined by the left action on $V^{\otimes n}$ wich permutes the indices:
\begin{equation}
(f.\sigma)(v_1 \otimes\cdots\otimes v_n )=f(v_{\sigma^{-1}(1)}\otimes\cdots\otimes v_{\sigma^{-1}(n)}),~~~~~~\forall~~f\in Endop(V)(n).
\end{equation}
For any $f\in Endop(V)(n),~~~g\in Endop(V)(m)$ and for any $i\in\{1,\cdots,n\}$,we defined partial composition, denoted by $f\circ_i g$, as the  morphism in $Endop(V)(n+m-1)$ defined by:
\begin{equation}
(f\circ_i g) (v_1 ,\cdots,v_{n+m-1})=f\big(v_1 ,\cdots,v_{i-1},g(v_i ,\cdots,v_{i+m-1} ),v_{i+m},\cdots,v_{m+n-1}\big).
\end{equation}
The unit map is the identity $Id_V$ for the space $V$. The partial compositions just defined satisfy the axioms of an operad. 
\end{exam}
\begin{defn}
Let $\mathcal{P}$ and $\mathcal{Q}$ be two operads. An morphism of operads from  $\mathcal{P}$ to $\mathcal{Q}$ is a  sequence
$a=\{a(n),~n\in\mathbb{N}^* \}$ of $K[\sum_{n}{}]$-linear maps
$a(n):\mathcal{P}(n)\rightarrow\mathcal{Q}(n)$ such that:
\begin{enumerate}
\item $a(1)(\un)=\un,$
\item $\forall n,m\in\mathbb{N}^*
  ,i\in\{1,\ldots,n\},~\mu\in\mathcal{P}(n),~\nu\in\mathcal{P}(m),$
\begin{equation*}
a(n+m-1)(\mu\circ_i \nu)=a(n)(\mu)\circ_i a(m)(\nu).
\end{equation*}
\end{enumerate}
\end{defn}
\begin{defn}
Let $V$ be a $K$-vector space and $\mathcal{P}$ an operad. The space
$V$ is a $\mathcal{P}$-algebra or algebra on $\mathcal{P}$, if there exists
a morphism of operads $a$  from $\mathcal{P}$ to the operad
$Endop(V)$. For any $\mu\in\mathcal{P}(n),~~v_1 ,\ldots,v_n \in V,$ we use the simplified notation:
\begin{equation}\label{equationpalgebre}
\mu(v_1 ,\ldots,v_n )=a(n)(\mu)(v_1 \otimes\ldots\otimes v_n ).
\end{equation}
Let $(V,a)$ and $(W,b)$ be two $\mathcal{P}$-algebras. Let $\phi:
(V,a)\rightarrow (W,b)$ be a linear map.  $\phi$ is called an morphism
of $\mathcal{P}$-algebras if for any $\mu\in\mathcal{P}(n),~~v_1
,\ldots,v_n \in V$,
\begin{equation}\label{morphismepalgeb}
\phi\big(\mu(v_1 ,\ldots,v_n )\big)=\mu\big(\phi(v_1 ),\ldots,\phi(v_n )\big)
\end{equation}
\end{defn}
\begin{defn}{\bf{ The free $\mathcal{P}$-algebra on a vector space}}\label{freeal}\\
Let $\mathcal{P}$ be an operad and $V$ be a  vector space. We define
the free $\mathcal{P}-$algebra on $V$, denoted by
$\mathcal{F}_{\mathcal{P}}(V)$ by the following universal property:\\
There exists $i:V\rightarrow\mathcal{F}_{\mathcal{P}}(V)$ such that for any
$\mathcal{P}-$algebra $A$ and for any linear map $\varphi:
V\rightarrow A$, there exists a unique morphism of
$\mathcal{P}-$algebras
$\tilde{\varphi}:\mathcal{F}_{\mathcal{P}}(V)\rightarrow A$ such that the following diagram commutes:
 \diagramme{
\xymatrix{\mathcal{F}_{\mathcal{P}}(V)\ar[r]^{\tilde{\varphi}}&A\\
V\ar[u]_{i}\ar[ur]_{\varphi}&
}
}
\end{defn}
\begin{rmk}
The free $\mathcal{P}-$algebra on $V$ is unique to up isomorphism.
\end{rmk}
\begin{prop}\label{proposition libre}
For any operad $\mathcal{P}$ and any vector space $V$, we have:
\begin{equation}
\mathcal{F}_{\mathcal{P}}(V)\approx
\bigoplus_{n\in\mathbb{N}^*}\mathcal{P}(n)\otimes_{\sum_{n}{}}V^{\otimes n}.
\end{equation}
\end{prop}
\begin{proof}
Let $\mathcal{P}$ be an  operad and   $V$  be a vector space. Let $A$ be a
 $\mathcal{P}-$algebra and $\varphi :V\rightarrow A$ a linear map. We set
$\mathcal{F}=\bigoplus_{n>0}\mathcal{P}(n)\otimes_{\sum_{n}{}}V^{\otimes
  n}$. Firstly  we prove that is  $\mathcal{F}$ is a
$\mathcal{P}-$algebra. For any $n\in\mathbb{N}^*$, we define:
$$\begin{array}{ccccl}
a(n)&:&\mathcal{P}(n)&\longrightarrow&Hom(\mathcal{F}^{\otimes
  n},\mathcal{F})\\
    & &\mu           &\longmapsto    & a(n)(\mu)
\end{array}$$
where for any $\nu_1 \in\mathcal{P}(k_1),\ldots,\nu_n \in\mathcal{P}(k_n),v_1
\in V^{\otimes k_1},\ldots,v_n \in V^{\otimes k_n}$,
\begin{equation}
a(n)(\mu)\big(\nu_1 \otimes v_1 ,\ldots,\nu_n \otimes v_n
\big)=\gamma(\mu,\nu_1 ,\ldots,\nu_n )\bigotimes(v_1 \otimes\ldots\otimes v_n ).
\end{equation}
$a$ is an morphism of operads. Also,
$$\begin{array}{ccl}
a(1)(\un)(\nu_1 \otimes v_1)&=&\gamma(\un,\nu_1)\otimes v_1\\
                              &=&\nu_1\otimes v_1
\end{array}$$
showing $a(1)(\un)=\un.$ Let
$n,m\in\mathbb{N}^*,\mu\in\mathcal{P}(n), \xi\in\mathcal{P}(m),
i\in\{1,\ldots,n\},~\nu_1 \in
\mathcal{P}(k_1),\ldots,\nu_{n+m-1}\in\mathcal{P}(k_{n+m-1}),v_1\in V^{\otimes
k_1},\ldots,v_{n+m-1}\in V^{\otimes n+m-1}$. We have:
 \begin{eqnarray*}
&&a(n)(\mu)\circ_i a(m)(\xi)\big(\nu_1 \otimes v_1,\ldots,\nu_{n+m-1}\otimes
v_{n+m-1}\big)\\
&=&a(n)(\mu)\left(\nu_1\otimes v_1,\ldots,\nu_{i-1}\otimes
v_{i-1},a(m)(\xi)(\nu_i \otimes v_i ,\ldots,\nu_{i+m-1}\otimes
v_{i+m-1}\big),\right.\\
 & &\left. \nu_{i+m}\otimes v_{i+m},\ldots,\nu_{n+m-1}\otimes
v_{n+m-1}\right.\big)\\
&=&a(n)(\mu)\left(\nu_1 \otimes v_1,\ldots,\nu_{i-1}\otimes
v_{i-1},\gamma(\xi,\nu_{i},\ldots,\nu_{i+m-1})\otimes \big(v_i \otimes\ldots\otimes
v_{i+m-1}\big),\right.\\
& &\left. \nu_{i+m}\otimes v_{i+m},\ldots,\nu_{n+m-1}\otimes
v_{n+m-1}\right)\\
&=&\gamma\big(\mu,\nu_1,\ldots,\nu_{i-1},\gamma(\xi,\nu_i,\ldots,\nu_{i+m-1}),\nu_{i+m},\ldots,\nu_{n+m-1}\big)\bigotimes
v_1 \otimes\ldots\otimes v_{n+m-1}\\
&=&\gamma\big(\mu\circ_i \xi,\nu_1,\ldots,\nu_{n+m-1}\big)\bigotimes
v_1\otimes\ldots\otimes v_{n+m-1}\\
&=&a(n+m-1)(\mu\circ_ i \xi)\big( \nu_1 \otimes v_1 ,\ldots,\nu_{n+m-1}\otimes v_{n+m-1}\big),
\end{eqnarray*}
which proves $a(n+m-1)(\mu\circ_i \xi)=a(n)(\mu)\circ_i a(m)(\xi)$. We define
$\tilde{\varphi}:\mathcal{F}\rightarrow A$ as follows: for any
$n>0,~~\mu\in\mathcal{P}(n)$ and $v_1 ,\ldots,v_n \in V,$ 
\begin{equation}
\tilde{\varphi}(\mu,v_1 \otimes\ldots\otimes v_n)=\mu\big( \varphi(v_1 ),\ldots,\varphi(v_n )\big),
\end{equation}
using the notation \eqref{equationpalgebre}. Since by construction $\tilde{\varphi}$ is an morphism of
$\mathcal{P}-$algebras \eqref{morphismepalgeb}, we have: 
\begin{equation}
\tilde{\varphi}\big(a(n)(\mu)(\nu_1\otimes v_1 ,\ldots,\nu_n\otimes
v_n\big)=\mu\big(\tilde{\varphi}(\nu_1\otimes
  v_1),\ldots,\tilde{\varphi}(\nu_n\otimes v_n)\big)
\end{equation}
which proves the  uniqueness of $\tilde{\varphi}$ by induction on the
graduation of $\mathcal{F}$.  We define:
$$\begin{array}{ccccl}
i&:&V&\longrightarrow&\mathcal{F}\\
 & &v&\longmapsto    & \un\otimes v
\end{array}$$
where $\un$ is the unit in the operad $\mathcal{P}$. It is clear than:
\begin{eqnarray*}
\tilde{\varphi}\circ i(v)&=&\tilde{\varphi}(\un\otimes v)\\
                         &=&\un\big(\varphi(v)\big)\\
                         &=&\varphi(v).
\end{eqnarray*}
  Then $\mathcal{F}$ satisfies the universal property above, which proves the proposition \ref{proposition libre}.
\end{proof}

 \section{Operads: an approach by species}
 The material presented here is mostly borrowed from \cite{Chap2}.
 
 \subsection{Categories}
 A category $\mathcal{C}$ is a collection of objects $Obj(\mathcal{C})$ and the data for all pairs of objects $(A,B)$ of a collection of morphisms (or arrows) from $A$ to $B$, which verify the axioms of a category \cite{SML}.
 \begin{exam}
 \begin{enumerate}
 \item Category of finite sets: objects are finite sets and arrows are the bijections.
 \item Category of vector spaces: objects are the vector spaces and arrows are linear maps.
 \end{enumerate}
 \end{exam}

 \subsection{Functors}

 A covariant functor  $\mathcal{F}$ is a "morphism" between two categories $\mathcal{C}$ and $\mathcal{D}$. Specifically, for every object $A$ of $\mathcal{C},~~~\mathcal{F}(A)$ is an object of $\mathcal{D}$ and for every arrow $f$ of $\mathcal{C},~~\mathcal{F}(f)$ is an arrow of $\mathcal{D}$ satisfying:
 $$\left\{\begin{array}{ccl}
 \mathcal{F}(Id_A )&=&Id_{\mathcal{F}(A)}~~\forall~A\in Obj(\mathcal{C})\\
 \mathcal{F}(g\circ f)&=&\mathcal{F}(g)\circ\mathcal{F}(f)~~~~\forall f,g~~\text{composable arrows of}~~\mathcal{C}
 \end{array}\right.$$

 \subsection{Species}

 Let $\mathcal{C}$ be a symmetric tensor category  \cite{SML}. A species in the category
  $\mathcal{C}$ \cite{AJ} is a functor
 $\mathcal{F}$ from the category of finite sets $\mathcal{F}_{in}$ to
 $\mathcal{C}$. Thus,  a species $\mathcal{F}$ provides an object
 $\mathcal{F}(A)$ for any finite set $A$ and an isomorphism
 $\mathcal{F}_{\varphi}: \mathcal{F}(A)\rightarrow\mathcal{F}(B)$ for any
 bijection $\varphi:A\rightarrow B$. In particular $Aut~A\approx S_n$ acts on
  $\mathcal{F}(A)$ if the cardinal of $A$ is equal to $n$. Let
 $\mathcal{F}$ and $\mathcal{G}$ two species. A morphism of
 species between $\mathcal{F}$ and $\mathcal{G}$ is a natural transformation
  $\psi :\mathcal{F}\longrightarrow \mathcal{G}$. So
 for any finite sets  $A$ and $B$ of the same cardinal and any
 bijection $\varphi$ from $A$ to $B$, the following diagram commutes:
 \diagramme{
\xymatrix{\mathcal{F}(A)\ar[r]^{\mathcal{F}_{\varphi}}\ar[d]_{\psi(A)}
&\mathcal{F}(B)\ar[d]^{\psi(B)}\\
\mathcal{G}(A)
\ar[r]_{\mathcal{G}_{\varphi}}&\mathcal{G}(B)
}
}

 \subsection{Species composition}

 Let $I$ be a finite set. Any equivalence relation  $R$ on  $I$ induces a partition of
 $I$ into blocks. Let $R$ and $R'$ two equivalence
 relations on $I$. We say that $R$ is finer than $R'$ and we write
 $R<R'$ if for any $x,y\in I,~~ xRy$ implies $xR'y$. Evidently any block $R$ is contained in a block of $R'$. We denote by $\mathcal{E}_{\mathcal{C}}$ the  category of species in $\mathcal{C}$. The species  composition is a  bifunctor  $\boxtimes:\mathcal{E}_{\mathcal{C}}\times\mathcal{E}_{\mathcal{C}}\rightarrow\mathcal{E}_{\mathcal{C}}$ such that for any finite set $I$,
 \begin{equation}\label{foncteur}
 (F\boxtimes G)(I)=\bigoplus_{R\models I}F(I/R)\otimes \bigotimes_{J\in I/R}G(J),
 \end{equation}
 where the notation $R\models I$ means that  $R$ is an equivalence relation on  $I$. $Aut ~I$ acts on the equivalence relations on $I$ and hence on $(F\boxtimes G)(I)$.
 \begin{rmk}
 The order on the different $J\in I/R$ is not specified in writing the tensor product in the right-hand side of Equation
 \eqref{foncteur}. The passage of a chosen sequence to another is an isomorphism given by the repeated use of flips $\tau_{AB}:A\otimes B\rightarrow B\otimes A.$
 \end{rmk}
 \begin{prop}
 The product $\boxtimes$ defined above is associative.
 \end{prop} 
 \begin{proof}
 Let $I$ be a finite set, $F,G$ and $H\in \mathcal{E}_{\mathcal{C}}$, then
 \begin{eqnarray*}
 \big((F\boxtimes G)\boxtimes H\big)(I)&=&\bigoplus_{R\models I}(F\boxtimes G)(I/R)\otimes\bigotimes_{J\in I/R}H(J)\\
 & =&\bigoplus_{R\models I}\bigg(\bigoplus_{\mathcal{S}'\models I/R}F((I/R)/\mathcal{S}')\otimes \bigotimes_{K'\in (I/R)_{/\mathcal{S}'}}G(K')\bigg)\otimes\bigotimes_{J\in I/R}H(J)\\
 &=&\bigoplus_{R\models I}\bigg(\bigoplus_{\mathcal{S}\models I, R<\mathcal{S}}F(I/\mathcal{S})\otimes\bigotimes_{K\in I/\mathcal{S}}G(K/R)\bigg)\otimes\bigotimes_{J\in I/R}H(J)\\
 &=&\bigoplus_{R<\mathcal{S}\models I}F(I/\mathcal{S})\otimes\bigotimes_{K\in I/\mathcal{S}} G(K/R)\otimes\bigotimes_{J\in I/R}H(J).
 \end{eqnarray*}
 While:
 \begin{eqnarray*}
 \big(F\boxtimes(G\boxtimes H)\big)(I)&=&\bigoplus_{\mathcal{S}\models I}F(I/\mathcal{S})\otimes\bigotimes_{J\in I/\mathcal{S}}(G\boxtimes H)(J)\\
 &=&\bigoplus_{\mathcal{S}\models I}F(I/\mathcal{S})\otimes\bigotimes_{J\in I/\mathcal{S}}\bigg(\bigoplus_{R\models J}G(J/R)\otimes\bigotimes_{K\in J/R}H(K)\bigg)\\
 &=&\bigoplus_{R<\mathcal{S}\models I}F(I/\mathcal{S})\otimes\bigotimes_{K\in I/\mathcal{S}}G(K/R)\otimes\bigotimes_{J\in I/R}H(J),
 \end{eqnarray*}
 Which proves the proposition. In the passage to the last line, we used the fact that giving an equivalence relation on each block of  $\mathcal{S}$ is equivalent to considering an equivalence relation $R<\mathcal{S}$ on $I$. 
 \end{proof}
 \begin{rmk}
 The product $\boxtimes$ behaves as a tensor product, apart from the fact that $F\boxtimes G\neq G\boxtimes F$ in general. The neutral element is the species $E$ such that $E(\{*\})=\un$ and $E(I)=0$ for any $I$ of cardinality greater than  $2$ or the empty set $I$.
 \end{rmk}
 {\bf {Notations :}} Let $I$ be a finite set and  $R$ an equivalence relation on $I$. We set $(F\boxtimes F)(I)_R =F(I/R)\otimes\bigotimes_{J\in I/R} F(J)~$, so that we have: $(F\boxtimes F)(I)=\bigoplus_{R\models I}(F\boxtimes F)(I)_R$. For  an automorphism $\sigma$ of $I$ and $J$  block of $I/R$, the restriction $\sigma\restr{J}$ is a bijection from $J$ onto $J.\sigma$. We denote by  $R^{\sigma}$ the equivalence relation on $I$ defined by $iRj\Leftrightarrow \sigma(i)R^{\sigma}\sigma(j)$ and we denote by $\bar{\sigma}$ the isomorphism of  $I/R$ onto $I/R^{\sigma}$  deduced from $\sigma$.
 
 \subsection{Background on operads}
 
  \begin{defn}
 An operad (in the category $\mathcal{C}$) is a monoid in the category
 $\mathcal{E}_{\mathcal{C}}$. Specifically, an operad is a species $F$ with a morphism
 $\gamma:F\boxtimes F\rightarrow F$ which is associative, i.e 
 $$\gamma(\gamma\boxtimes I)=\gamma(I\boxtimes \gamma).$$
  The morphism $\gamma$ defines for any set $I$ and any equivalence relation $R$ on $I$:
 \begin{equation}
 \gamma_{I,R} :F(I/R)\otimes \bigotimes_{J\in I/R}F(J)\longrightarrow F(I).
 \end{equation}
 The equivariance resumes in the following commutative diagram:
 \diagramme{
\xymatrix{F(I/R)\otimes\bigotimes_{J\in I/R}F(J)\ar[rrr]^{\gamma_{I,R}}\ar[d]_{(F\boxtimes F)(\sigma)\restr{(F\boxtimes F)(I)_R}}
&&&F(I)\ar[d]^{F(\sigma)}\\
F(I/R^{\sigma})\otimes\bigotimes_{J\in I/ R^{\sigma}}F(J)
\ar[rrr]_{\gamma_{I,R^{\sigma}}}&&&F(I)
}
}
where $(F\boxtimes F)(\sigma)=F(\bar{\sigma})\otimes\bigotimes_{J\in I/R}F(\sigma\restr{J})$.   
 \end{defn}
 
 \subsection{Partial composition}

 For simplicity we can assume that the operad $F$ is augmented i.e: $F(\emptyset)=0$ and $F(\{*\})=\un$. We consider on the finite set $I'$ an equivalence relation $R$ such that all blocks are singletons but one denoted by $J$. We set $I=I'/R$ and $i=\{J\}\in I$, hence $I'=I\backslash\{i\}\cup J$. We denote by  $\circ_i$ the composition: $\gamma:F(I'/R)\otimes F(J)\rightarrow F(I')$, hence $\gamma:F(I)\otimes F(J)\rightarrow F(I\backslash\{i\}\cup J)$. Other components are left of the form $F(\{*\})=\un$, so do not appear in the tensor product. Associativity is expressed by the commutativity of the two following diagrams:
 \diagramme{
\xymatrix{F(I)\otimes F(J)\otimes F(K)\ar[r]^{Id_I \otimes \circ_j}\ar[d]_{\circ_i}
&F(I)\otimes F(J\backslash\{j\}\cup K)\ar[d]^{\circ_i}\\
F(I\backslash\{i\}\cup J)\otimes F(K)
\ar[r]_{\circ_j}&F(I\backslash\{i\}\cup J\backslash\{j\}\cup K)
}
}
and
\diagramme{
\xymatrix{F(I)\otimes F(J)\otimes F(K)\ar[r]^{\circ_i \otimes Id}\ar[d]_{\tau^{2,3}}
&F(I\backslash\{i\}\cup J)\otimes F(K)\ar[dd]^{\circ_l}\\
F(I)\otimes F(K)\otimes F(J)\ar[d]_{\circ_l \otimes Id}&\\
F(I\backslash\{l\}\cup K)\otimes F(J)
\ar[r]_{\circ_i}&F(I\backslash\{i,l\}\cup J\cup K)
}
}
where $\tau^{2,3}=Id_{F(I)}\otimes\tau_{F(J),F(K)}$. The first diagram corresponds to the nested associativity while the second expresses disjoint associativity.

\section{Two operads of rooted trees: pre-Lie and NAP}

We recall briefly in this section the pre-Lie operad of rooted trees introduced by F. Chapoton and M. Livernet in  \cite{ChaLiv}. A kind of simplified version of it is the NAP operad introduced by M. Livernet \cite{Liv}. A rooted tree is a connected graph without loops, one of whose vertices is not departure of any edge; this vertex is called the root. Rooted trees are drawn with the root down (see \cite{F},\cite{CK1}). For all $n\geq1$, a tree of degree $n$ is a tree with $n$ vertices labelled from  $1$ to $n$. The set $\{1,\ldots,n\}$ is denoted $[n]$. We note $\mathcal{RT}(n)$ the space of labelled rooted trees  of degree $n$. Let $\mathcal{RT}=\bigoplus_{n\geq1}\mathcal{RT}(n)$. We endow $\mathcal{RT}$ with a structure of operad (see \cite{ChaLiv}): the action of the symmetric group is natural by permuting the labels of the vertices. For    a tree $T$ of degree $n$ and for all $i\in[n]$, we denote by  $E(T,i)$
the set of edges of  $T$ arriving at vertex $i$ of $T$.

\subsection{The pre-Lie operad}

We define the partial compositions $\circ_i ~:~\mathcal{RT}(n)\otimes\mathcal{RT}(m)\rightarrow\mathcal{RT}(n+m-1)$, as follows:
\begin{equation}
T\circ_i S=\sum_{f: E(T,i)\rightarrow [m]}{T\circ_{i}^{f}S},
\end{equation}
where $T\circ_{i}^{f}S$ is the tree of  $\mathcal{RT}(n+m-1)$ obtained by replacing the vertex $i$ of $T$ by the tree $S$ and connecting each edge $a$
in $E(T,i)$ at the vertex $f(a)$ of $S$. This new tree is labelled as follows: we add $i-1$ to those of $S$ and $m-1$ to those of $T$ that are greater than $i+1$. The root of the new tree is the root of  $T$ if it is different from vertex $i$, and of  $S$ else (see detail in
 \cite{ChaLiv}). The unit is the tree with a single vertex. These partial compositions define an operad which is the pre-Lie operad. We define: $T\vartriangleleft
S=\sum_{i=1}^{n}{(T\circ_{i}S)}$. The law $\vartriangleleft$ passes to the quotient by the symmetric groups, according to the equivariance axiom.
\begin{thm}
The space $\mathcal{F}_{PL}=\bigoplus_{n\geq1}\mathcal{RT}(n)/S_n$ equipped with the bilinear map
$\vartriangleleft$ defined above is a  right pre-Lie algebra. this theorem is true for any  augmented operad $\mathcal{O}$ i.e : such that  $\dim \mathcal{O}_1 =1$ \cite{ChaLiv}: namely $\mathcal{F}_{\mathcal{O}}:=\mathcal{F}_{\mathcal{O}}(V)$, with the notation of Definition \ref{freeal}, is a right pre-Lie algebra.
\end{thm}
\begin{rmk}
The partial compositions defined above give rise to two right pre-Lie structures $\vartriangleleft$ and  $\leftarrow$ on $\mathcal{F}_{PL}$  \cite{ChaLiv}, the first acting on the second by derivation \cite{MS}. The pre-Lie structure $\leftarrow$ is defined by:
\begin{equation}
T\leftarrow S=\gamma(\echelundeux,T,S),
\end{equation}
called the grafting $S$ on $T$ and we have the following  derivation relation:
\begin{equation}
(T\leftarrow S)\vartriangleleft U=(T\vartriangleleft U)\leftarrow S+T\leftarrow(S\vartriangleleft U)~\forall~~ S,T,U\in \mathcal{F}_{PL}.
\end{equation}
\end{rmk}

\subsection{The NAP operad}

We define \cite{Liv} the partial compositions $\circ_i ~:~\mathcal{RT}(n)\otimes\mathcal{RT}(m)\rightarrow\mathcal{RT}(n+m-1)$, as follows:
\begin{equation}
T\circ_i S={T\circ_{i}^{f_0}S},
\end{equation}
where $T\circ_{i}^{f_0}S$ is the tree of  $\mathcal{RT}(n+m-1)$ obtained by replacing the vertex $i$ of $T$ by the tree $S$ and connecting each edge $a$
in $E(T,i)$ at the root of $S$. 
  The unit is the tree with a single vertex. We define: $T\odot
S=\sum_{i=1}^{n}{(T\circ_{i}S)}$. The law $\odot$ passes to the quotient by the symmetric groups, according to the equivariance axiom.
\begin{thm}
The space $\mathcal{F}_{NAP}=\bigoplus_{n\geq1}\mathcal{RT}(n)/S_n$ equipped with the bilinear map
$\odot$ defined above is a  right pre-Lie algebra.  
\end{thm}
\begin{rmk}
The partial compositions defined above also gives rise to a right NAP structure $\swarrow$ on $\mathcal{F}_{NAP}$. The NAP structure $\swarrow$ is defined by:
\begin{equation}
T\swarrow S=\gamma(\echelundeux,T,S),
\end{equation}
called the (right) Butcher product of $T$ and $S$, and we have the following  derivation relation:
\begin{equation}
(T\swarrow S)\odot U=(T\odot U)\swarrow S+T\swarrow(S\odot U)~\forall~~ S,T,U\in \mathcal{F}_{NAP}.
\end{equation}
\end{rmk}

 \section{The notion of multigraded operad}

 This section defines the notion of multigraded operad. The model of which is the operad $Endop(V)$ where $V$ is a graded vector space. Multigraded operads are closely related to colored operads, except that a semi-group structure on the set of colors is taken into account.

 \subsection{ General definition}\label{multigraduee}

  Let $\mathcal{O}=\bigoplus_{n\geq1}\mathcal{O}_n$ be an  operad. We say that $\mathcal{O}$ has a structure of multigraded operad, if moreover $\mathcal{O}_n =\prod\mathcal{O}_{n, a_1 ,\ldots,a_n},~~a_j\in\mathbb{N}^*$ where:
 \begin{itemize}
 \item The right action of the symmetric group is verifies:
 \begin{equation}
 (\mathcal{O}_{n,a_1 ,\ldots,a_n}).\sigma=\mathcal{O}_{n,a_{\sigma(1)},\ldots,a_{\sigma(n)}}~~\forall~\sigma \in S_n~.
 \end{equation}
 \item For any $n,m\in\mathbb{N}^*$, we have:
 \begin{equation}
 \circ_i :\mathcal{O}_{n,a_1 ,\ldots,a_n}\otimes\mathcal{O}_{m,b_1 ,\ldots,b_m}\rightarrow\mathcal{O}_{m+n-1,a_1 ,\ldots,a_{i-1},b_1,\ldots,b_m ,a_{i+1},\ldots,a_n},
 \end{equation}
 with have image zero if $\sum_{j=1}^{m}{b_j}\neq a_i$.
 \end{itemize}

\subsection{Example: the multigraded operad $Endop(V)$ for a graded vector space $V$}

 Let $V=\bigoplus_{j\geq0} V_j$ be a graded vector space. Let $Endop(V)(n)$ be the space of homogenous degree zero morphisms from $V^{\otimes n}$ to $V$, i.e:
\begin{eqnarray*}
Endop(V)(n)&=&\mathcal{L}_0 (V^{\otimes n},V)\\
           &=&\prod_{a_1 ,\ldots,a_n , n\geq1}\mathcal{L}(V_{a_1}\otimes\ldots\otimes V_{a_n},V_{a_1 +\ldots+a_n}).
 \end{eqnarray*}
 Thus, for any  $\alpha\in\mathcal{L}_0 (V^{\otimes n},V)$ and for any $a_1,\ldots,a_n\in \mathbb{N}$, we define $\alpha_{a_1 ,\ldots,a_n}$ by: 
 $$\left\{\begin{array}{ccl}
 \alpha_{a_1 ,\ldots,a_n}(v_1 \otimes\ldots\otimes v_n)&=&\alpha(v_1 \ldots v_n )~~~\text{if }~~v_j\in V_{a_j} ~~\forall~j\in\{1,\ldots,n\}\\
  &=&0~~\text{else}.\end{array}\right.$$
  The partial compositions  $\circ_i$ satisfy:
  \begin{equation}
 (\alpha\circ_i \beta)_{c_1 ,\ldots,c_{n+m-1}}=\alpha_{c_1,\ldots,c_{i-1},\left(\sum_{j=i}^{i+m-1}{c_j}\right),c_{i+m},\ldots,c_{m+n-1}}\circ_i \beta_{c_i ,\ldots,c_{i+m-1}},
 \end{equation}
  and therefore 
 \begin{equation} 
 \alpha_{a_1 ,\ldots,a_n}\circ_i \beta_{b_1 ,\ldots,b_m}=0,
 \end{equation}
 except perhaps if $a_i =\sum_{j=1}^{m}{b_j}$. Thus:
 \begin{thm}
 $Endop(V)$ is a multigraded operad in the sense of paragraph \ref{multigraduee}
 \end{thm}

 \section{The multigraded operad $\mathcal{O}^{\lambda}$}

 \begin{defn}
 We define trees with weights on their vertices:
 $\racineun ,\racinedeux ,\racinetrois ,\ldots$.
 \begin{exam}
 $\racineun ,\racinedeux ,\racinetrois ,\echelunun,\echelundeux,\echeldeuxun,\couronne,\echelununun$ are the trees with weight less or equal to $3$.
 \end{exam} 
 For a tree $T$, we define the weight of  $T$ by:
 \begin{equation}
 \left|T\right|=\sum_{v\in v(T)}{\left|v\right|},
 \end{equation}
 where $v(T)$ denotes the set of vertices of $T$, and $\left|v\right|$ is the weight of $v$. 
 \end{defn}
 We draw trees with labels and numbers on their vertices, each number refers to the weight of the vertex.
 \begin{defn}
 We define the potential energy of a tree by:
 \begin{equation}
 d(T)=\sum_{v\in v(T)}{\left|v\right|h(v)},
 \end{equation}
 where $h(v)$ is the height of $v$ in $T$, i.e. the distance from $v$ to the root of $T$ counting the number of edges.
 \end{defn}
Let $\lambda$ be an element of the field $K$. For any  $n\in \mathbb{N}^*$, let $\mathcal{O}_{n}^{\lambda}$ the vector space spanned by the trees with $n$ vertices of any weight. Let $\mathcal{O}^{\lambda}=\bigoplus_{n\geq1} \mathcal{O}_{n}^{\lambda}$. For all $S\in \mathcal{O}_{n}^{\lambda}$ and any vertex $v$ of $S$, $E(S,v)$ denotes the set of edges of $S$ arriving at the vertex  $v$ of $S$. Let $T\in \mathcal{O}_{m}^{\lambda}$, we define the partial compositions by:
$$S\circ_{v,\lambda} T=\left\{\begin{array}{l}
\sum_{f:E(S,v)\rightarrow v(T)}{\lambda^{d(S\circ_{v}^{f} T)-d(S\circ_{v}^{f_0}T)}S\circ_{v}^{f}T}~~~~\text{si}~~\left|~T\right|=\left|v\right|\\
 0~~~~~\text{otherwise},
\end{array}\right.$$
where  $S\circ_{v}^{f}T$ is the element of $\mathcal{O}_{n+m-1}^{\lambda}$ obtained by replacing the vertex  $v$ by $S$ and connecting each edge of $E(S,v)$ to its
 image by $f$ in $v(T)$. Here  $f_0$ is the map from $E(S,v)$ to
$v(T)$ that sends each edge $a$ of $E(S,v)$ to the root of $T$.
The tree $S\circ_{v}^{f_0}T$  has, therefore the smallest potontial energy in the above sum. Unity is given by:
$$\un=\sum_{n\geq1}{\racinen},$$
where $\racinen$ is the tree with one single vertex of weight $n$. The action of the symmetric group  $S_n$ on $\mathcal{O}_{n}^{\lambda}$ is given by permutation of the labels.
\begin{exam}
For $S=\arbaa$ and $T=\arbab$. There letters $a,b,c...$ are labels of vertices, which are of weight $1,2$ or $3$.  We have:
 \begin{align*}
S\circ_{b,\lambda} T&=&\arbac&+\lambda\arbad&+\lambda^{2}\arbae\\
          & &\hskip -8mm f_0
          :\left\{\begin{array}{ccc}c&\rightarrow&e\\d&\rightarrow&e\end{array}\right.\hbox to 10mm{}&f:\left\{\begin{array}{ccc}c&\rightarrow&e\\d&\rightarrow&h\end{array}\right.&f:\left\{\begin{array}{ccc}c&\rightarrow&h\\d&\rightarrow&e\end{array}\right.\\
        & &      &              &\\
&&         +\lambda^{3}\arbaf&&\\
&&f:\left\{\begin{array}{ccc}c&\rightarrow&h\\d&\rightarrow&h\end{array}\right.&&
\end{align*}
\end{exam}
\begin{thm}
the partial compositions defined above on  $\mathcal{O}^{\lambda}$  yield a structure of multigraded operad as defined in the previous paragraph.
\end{thm}
\begin{proof}
\begin{itemize}
\item Nested associativity: show $(S\circ_{v,\lambda} T)\circ_{w,\lambda} U=S\circ_{v,\lambda} (T\circ_{w,\lambda} U)$ where $v$ is a vertex of $S$ and $w$ a vertex of $T$.\\
Let $S,T,U \in \mathcal{O}^{\lambda}, v$ be a vertex of $S$ and $w$ be a vertex of
 $T$ such that $\left|T\right|=\left|v\right|$ and
$\left|U\right|=\left|w\right|$. We have:
\begin{eqnarray*}
(S\circ_{v,\lambda} T)\circ_{w,\lambda} U&=&\sum_{f:E(S,v)\rightarrow v(T)}{\lambda^{d(S\circ_{v}^{f}T)-d(S\circ_{v}^{f_0}T)}(S\circ_{v}^{f}T)\circ_{w,\lambda} U}\\
&\hskip -20mm=&\hskip -10mm \sum_{f:E(S,v)\rightarrow v(T)}{\sum_{g:E(S\circ_{v}^{f}T, w)\rightarrow v(U)}{\lambda^{d(S\circ_{v}^{f}T)-d(S\circ_{v}^{f_0}T)+d((S\circ_{v}^{f}T)\circ_{w}^{g}U)-d((S\circ_{v}^{f}T)\circ_{w}^{g_0 }U)}(S\circ_{v}^{f}T)\circ_{w}^{g}U}}\\
&=&\sum_{f:E(S,v)\rightarrow v(T)}{\sum_{g:E(S\circ_{v}^{f}T,w)\rightarrow v(U)}{\lambda^{A}(S\circ_{v}^{f}T)\circ_{w}^{g}U}},
\end{eqnarray*}
where $A=d(S\circ_{v}^{f}T)-d(S\circ_{v}^{f_0}T)+d\big((S\circ_{v}^{f}T)\circ_{w}^{g}U\big)-d\big((S\circ_{v}^{f}T)\circ_{w}^{g_0}U\big).$ Similarly we have:
\begin{eqnarray*}
S\circ_{v,\lambda}(T\circ_{w,\lambda}U)&=&\sum_{\tilde{g}:E(T,w):\rightarrow v(U)}{\lambda^{d(T\circ_{w}^{\tilde{g}}U)-d(T\circ_{w}^{\tilde{g_0}}U)}S\circ_{v,\lambda}(T\circ_{w}^{\tilde{g}}U)}\\
&\hskip -30mm=&\hskip -15mm \sum_{\tilde{f}:E(S,v)\rightarrow v(T\circ_{w}^{\tilde{g}}U}{\sum_{\tilde{g}:E(T,w)\rightarrow v(U)}{\lambda^{d(T\circ_{w}^{\tilde{g}}U)-d(T\circ_{w}^{\tilde{g_0}}U)+d\big(S\circ_{v}^{\tilde{f}}(T\circ_{w}^{\tilde{g}}U)\big)-d\big(S\circ_{v}^{\tilde{f}_{0}}(T\circ_{w}^{\tilde{g}}U)\big)}S\circ_{v}^{\tilde{f}}(T\circ_{w}^{\tilde{g}}U)}}\\
&=&\sum_{\tilde{g}:E(T,w)\rightarrow v(U)}{\sum_{\tilde{f}:E(S,v)\rightarrow v(T\circ_{w}^{\tilde{g}}U)}{\lambda^{B}}S\circ_{v}^{\tilde{f}}(T\circ_{w}^{\tilde{g}}U)}.
\end{eqnarray*}
In order to show $(S\circ_{v,\lambda} T)\circ_{w,\lambda} U=S\circ_{v,\lambda} (T\circ_{w,\lambda} U)$, we have to prove the following lemma:
\begin{lem}\label{lemme emboite}
there is a natural bijection $(f,g)\longmapsto (\tilde{f},\tilde{g})$ such that
\begin{equation}\label{equation emboite}
(S\circ_{v}^{f}T)\circ_{w}^{g}U=S\circ_{v}^{\tilde{f}}(T\circ_{w}^{\tilde{g}}U).
\end{equation}
\end{lem}
\begin{proof}
Let $v$ be a vertex of $S$ and $w$ be a vertex of $T$ such that $\left|T\right|=\left|v\right|$ and $\left|U\right|=\left|w\right|$. Let $f:E(S,v)\rightarrow v(T)$ and $g:E(S\circ_{v}^{f}T,w)\rightarrow v(U)$ us choose two applications. We look for $\tilde{g}:E(T,w)\rightarrow v(U)$ and $\tilde{f}:E(S,v)\rightarrow v(T\circ_{w}^{\tilde{g}}U)=v(T)\cup v(U)\backslash\{w\}$ such that the equation \eqref{equation emboite} is checked.\\
Let $e$ be an edge of $T$ arriving at $w$, thus $e$ is an edge of $S\circ_v T$ arriving at $w$. We set $\tilde{g}(e)=g(e)$. Similarly we define $\tilde{f}$ in a unique way:
$$\begin{array}{ccccl}
\tilde{f}&:&E(S,v)&\longrightarrow &V(T\circ_{w}^{\tilde{g}}U)=v(T)\cup v(T)\backslash\{w\}\\
         & & e    &\longmapsto     &\tilde{f}(e)=\left\{\begin{array}{ccc}f(e)& if &f(e)\neq w\\g(e)&if& f(e)=w \end{array}\right.
\end{array}$$ 
Conversely, we assume that we have the pair $(\tilde{f},\tilde{g})$ and look for the pair $(f,g)$ such that equation \eqref{equation emboite} is verified. We have $\tilde{f}:E(S,v)\rightarrow v(T)\cup v(U)\backslash\{w\}$ and $\tilde{g}:E(T,w)\rightarrow v(U).$ We then define:
$$\begin{array}{ccccl}
f&:&E(S,v)&\longrightarrow &v(T)\\
 & & e    & \longmapsto    &f(e)=\left\{\begin{array}{ccc}\tilde{f}(e)&if&\tilde{f}(e)\notin v(U)\\ w&if& \tilde{f}(e)\in v(U)\end{array}\right.
\end{array}$$
\end{proof}
To show $(S\circ_{v,\lambda} T)\circ_{w,\lambda} U=S\circ_{v,\lambda} (T\circ_{w,\lambda} U)$, it remains to show the equality $A=B$. We set $A'=A+d((S\circ_{v}^{f_0}T)\circ_{w}^{g_0}U)$,~~ $B'=B+d(S\circ_{v}^{\tilde{f}_{0}}(T\circ_{w}^{\tilde{g}_{0}}U))$ and $\epsilon(f)=d(S\circ_{v}^{f}T)-d(S\circ_{v}^{f_0}T)$. We have:
\begin{equation}
\epsilon(f)=\sum_{e\in E(S,v)}{h\big(f(e)\big).\left|B_e\right|},
\end{equation}
where $h\big(f(e)\big)$ is the distance between $f(e)$ and the root of $T$ in the new tree $S\circ_{v}^{f}T$ and $\left|B_e\right|$ is the weight of the branch to $e$. Similarly:
\begin{eqnarray*}
d\big((S\circ_{v}^{f}T)\circ_{w}^{g_0}U\big)-d\big((S\circ_{v}^{f_0}T)\circ_{w}^{g_0}U\big)&=&\sum_{e\in E(S,v)}{h\big(f(e)\big)\left|B_e\right|}\\
&=&\epsilon(f),
\end{eqnarray*}
Here $g_0$ is not involved because it was connected to the root of $U$, so $A'=d\big((S\circ_{v}^{f}T)\circ_{w}^{g}U\big)$. By the same calculation is shown $B'=d\big(S\circ_{v}^{\tilde{f}}(T\circ_{w}^{\tilde{g}}U)\big)$. So by Lemma \ref{lemme emboite} we have  $A'=B'$ that is to say $A+d\big((S\circ_{v}^{f_0}T)\circ_{w}^{g_0}U\big)=B+d\big(S\circ_{v}^{\tilde{f}_0}(T\circ_{w}^{\tilde{g}_0}U)\big)$, which proves that $A=B$ because by Lemma \ref{lemme emboite} we have $d\big((S\circ_{v}^{f_0}T)\circ_{w}^{g_0}U\big)=d\big(S\circ_{v}^{\tilde{f}_0}(T\circ_{w}^{\tilde{g}_0}U)\big).$
\item  Disjoint associativity: let $v,w$ be two disjoint vertices of $S$ such that $\left|v\right|=\left|T\right|$ and $\left|w\right|=\left|U\right|$, show that:
\begin{equation}\label{equation disjointe}
(S\circ_{v,\lambda} T)\circ_{w,\lambda} U=(S\circ_{w,\lambda} U)\circ_{v,\lambda} T.
\end{equation}
We have
\begin{eqnarray*}
(S\circ_{v,\lambda} T)\circ_{w,\lambda} U&=&\sum_{f:E(S,v)\rightarrow v(T)}{\lambda^{d(S\circ_{v}^{f}T)-d(S\circ_{v}^{f_0}T)}(S\circ_{v}^{f}T)\circ_{w,\lambda} U}\\
&=&\sum_{f:E(S,v)\rightarrow v(T)}{\sum_{g:E(S\circ_{v}^{f}T,w)\rightarrow v(U)}{\lambda^{k(f)+d\big((S\circ_{v}^{f}T)\circ_{w}^{g}U\big)-d\big((S\circ_{v}^{f}T)\circ_{w}^{g_0}U\big)}}(S\circ_{v}^{f}T)\circ_{w}^{g}U}\\
&=& \sum_{f:E(S,v)\rightarrow v(T)}{\sum_{g:E(S\circ_{v}^{f}T,w)\rightarrow v(U)}{\lambda^{C}}(S\circ_{v}^{f}T)\circ_{w}^{g}U},
\end{eqnarray*}
where $k(f)=d(S\circ_{v}^{f}T)-d(S\circ_{v}^{f_0}T)$ and
$C=k(f)+d\big((s\circ_{v}^{f}T)\circ_{w}^{g}U\big)-d\big((S\circ_{v}^{f}T)\circ_{w}^{g_0}U\big)$.
Similarly we find:
$$(S\circ_{w,\lambda} U)\circ_{v,\lambda}T=\sum_{\tilde{g}:E(S,w)\rightarrow v(U)}{\sum_{\tilde{f}:E(S\circ_{w}^{\tilde{g}}U,v)\rightarrow v(T)}{\lambda^{D}}(S\circ_{w}^{\tilde{g}}U)\circ_{v}^{\tilde{f}}T},$$
where $D=d(S\circ_{w}^{\tilde{g}}U)-d(S\circ_{w}^{\tilde{g}_0}U)+d\big((S\circ_{w}^{\tilde{g}}U)\circ_{v}^{\tilde{f}}T\big)-d\big((S\circ_{w}^{\tilde{g}}U)\circ_{v}^{\tilde{f}_{0}}T\big)$. We set $k(\tilde{g})=d(S\circ_{w}^{\tilde{g}}U)-d(S\circ_{w}^{\tilde{g}_0}U)$. To prove \eqref{equation disjointe}, we introduce the following lemma:
\begin{lem}\label{lemme disjointe}
We have a natural bijection $(f,g)\longmapsto (\tilde{f},\tilde{g})$ such that
\begin{equation}\label{equa}
(S\circ_{v}^{f}T)\circ_{w}^{g}U=(S\circ_{w}^{\tilde{g}}U)\circ_{v}^{\tilde{f}}T.
\end{equation} 
\end{lem}
\begin{proof}
Let $f:E(S,v)\rightarrow v(T)$ and $g:E(S\circ_{v}^{f}T,w)\rightarrow v(U)$ are given maps.  We look for $\tilde{g}:E(S,w)\rightarrow v(U)$ and $\tilde{f}:E(S\circ_{w}^{\tilde{g}}U,v)\rightarrow v(T)$ such that the Equation  \eqref{equa} is verified. Let  $\tilde{g}$ be the restriction of $g$ on the edges $e$ from $S$ and $\tilde{f}=f$ because $E(S\circ_{w}^{g}U,v)=E(S,v)$ because the vertices $v$ and $w$ are disjoint.
\end{proof}
Thus to show  disjoint associativity, it remains to show that for any pair $(f,g)$ and $(\tilde{f},\tilde{g})$ we have $C=D$. We set
$C'=C+d\big( (S\circ_{v}^{f_0}T)\circ_{w}^{g_0}U\big)$ and $D'=D+d\big((S\circ_{w}^{\tilde{g}_0}U)\circ_{v}^{\tilde{f}_0}T\big)$. We have:
$$k(f)=\sum_{e\in E(S,v)}{h\big(f(e)\big)\left|B_e\right|},$$
Similarly we have:
$$d\big((S\circ_{v}^{f}T)\circ_{w}^{g_0}U\big)- d\big((S\circ_{v}^{f_0}T)\circ_{w}^{g_0}U\big)=k(f),$$
because we changed the vertex $w$ by a tree of the same weight, and $g_0$ were grafted onto the root, which proves $C'=d\big((S\circ_{v}^{f}T)\circ_{w}^{g}U\big)$. So $D'=d\big((S\circ_{w}^{\tilde{g}}U)\circ_{v}^{\tilde{f}}T\big)$. By Lemma \ref{lemme disjointe} again, $C'=D'$ and  $C=D$, which proves disjoint associativity.
\end{itemize}
 The partial compositions defined on  $\mathcal{O}^{\lambda}$ hence verify the axioms of an multigraded operad.
\end{proof}
\begin{rmk}
The multigraded operad $\mathcal{O}^{\lambda}$ modulo the forgetting of the graduation
can be identified to:\\
- The pre-Lie operad \cite{ChaLiv} if $\lambda=1$.\\
- The NAP-operad if $\lambda=0$.
\end{rmk}
Let $\mathcal{O}=\bigoplus_{n\geq0}\mathcal{O}_n$ the multigraded operad. We have $\mathcal{O}_n =\prod_{i_j \geq1}\mathcal{O}_{n,i_1 ,\ldots,i_n}.$ for $k\in\mathbb{N}$. We have:
$$\mathcal{O}_{n}^{\lambda}=\{(t,\underline{\mu}),~~ \text{where t is a labeled tree with n vertices and}~~\underline{\mu}=(\mu_1 ,\ldots,\mu_n)\in{\mathbb{N}*}^{n}\},$$
where $\underline{\mu}$ encodes the weight of the vertices. We denote by
$\mathcal{O}_{PL}$ the pre-Lie operad of  labeled rooted trees   \cite{ChaLiv}. We set:
$$\begin{array}{ccccl}
i&:&\mathcal{O}_{PL}&\longrightarrow&\mathcal{O}^1\\
 & & t              &\longmapsto    & \sum_{\underline{\mu}\in\mathbb{N}^{nb(t)}}{(t,\underline{\mu})},
\end{array}$$ 
where $nb(t)$ denoted the number of vertices of the tree $t$. We denote by $i_n$ the restriction of $i$ to $\mathcal{O}_{PL,n}$. It is clear that for all $n\in\mathbb{N}^*$,
$$\begin{array}{ccccl}
i_n&:&\mathcal{O}_{PL,n}&\longrightarrow&\mathcal{O}_{n}^{1}\\
 & & t                &\longmapsto    &\sum_{\underline{\mu}\in\mathbb{N}^{n}}{(t,\underline{\mu})}.
\end{array}$$
\begin{thm}
$i$ is a morphism of operads.
\end{thm}
\begin{proof}
Let $S$ be a tree with $n$ labeled vertices and  $T$  be a tree with $m$ labeled vertices. Let $v$ be a vertex of $S$. Show:
\begin{equation}\label{morphismeoper}
i(S\circ_v T)=i(S)\circ_{v,1}i(T).
\end{equation}
Proving Equation  \eqref{morphismeoper} is equivalent to show  
\begin{eqnarray*}
\sum_{\underline{\gamma}\in{\mathbb{N}^*}^{n+m-1}}{(S\circ_v T,\underline{\gamma})}&=&\sum_{\underline{\alpha}\in{\mathbb{N}^*}^n}{(S,\underline{\alpha})}\circ_{v,1}\sum_{\underline{\beta}\in{\mathbb{N}^*}^m}{(T,\underline{\beta})}\\
&=&\sum_{\underline{\alpha}\in{\mathbb{N}^*}^n}{(S,\underline{\alpha})}\circ_{v,1}\sum_{\underline{\beta}\in{\mathbb{N}^*}^m /~\left|\beta\right|=\left|v\right|}{(T,\underline{\beta})}.
\end{eqnarray*}
Let $f:E(S,v)\longrightarrow v(T)$ and
$\underline{\gamma}=(\gamma_1,\ldots,\gamma_{n+m-1})\in{\mathbb{N}^{*}}^{n+m-1},$
 thus we can identify $\underline{\gamma}$ to
 \begin{equation}
\{\gamma_w,w~\text{vertex of}~~S\circ_{v}^{f}T\}.
\end{equation}
 As we have:
\begin{equation}
v(S\circ_{v}^{f}T)=v(S)\backslash\{v\}\cup v(T),
\end{equation}
 we set:
\begin{eqnarray*}
\underline{\beta}&=&\{\beta_w=\gamma_w ,~~~w\in v(T)\},\\
 and&&\\
\underline{\alpha}&=&\{\alpha_w =\gamma_w ,~~~~w\in
v(S)\backslash\{v\},\alpha_v =\sum_{w\in v(T)}{\beta_w}\}.
\end{eqnarray*}
Thus for each choice of $f$ and $\underline{\gamma}$, there exist
$\underline{\alpha}$ and $\underline{\beta}$ unique such that: 
\begin{equation}
(S\circ_{v}^{f}T,\underline{\gamma})=(S,\underline{\alpha})\circ_{v,1}^{f}(T,\underline{\beta}).
\end{equation}
Thus 
\begin{equation}
\sum_{\underline{\gamma}}{(S\circ_{v}^{f}T,\underline{\gamma})}=\sum_{\underline{\alpha},\underline{\beta}/\left|\alpha_v\right|
=\left|\underline{\beta}\right|}{(S,\underline{\alpha})\circ_{v,1}^{f}(T,\underline{\beta})}.
\end{equation}
As a result 
\begin{eqnarray*}
i(S\circ_{v}^{f}T)&=&\sum_{\underline{\alpha}\in{\mathbb{N}^*}^{n}}{(S,\underline{\alpha})}\circ_{v,1}^{f}\sum_{\underline{\beta}\in{\mathbb{N}^*}^{m}}{(T,\underline{\beta})}\\
&=& i(S)\circ_{v,1}^{f}i(T)
\end{eqnarray*}
Then summing over different possible connections $f$, we find:
\begin{equation}
i(S\circ_v T)=i(S)\circ_{v,1}i(T),
\end{equation}
which proves the theorem.
\end{proof}
If $\circ_{NAP}$ denoted the partial composition for the $NAP$ operad, we similarly prove that: 
$$\begin{array}{ccccl}
j&:&\mathcal{O}_{NAP}&\longrightarrow &\mathcal{O}^{0}\\
 & &t               &\longmapsto     &\sum_{\underline{\mu}\in{\mathbb{N}^{*}}^{nb(t)}}{(t,\underline{\mu})}
\end{array}$$
 is an morphism of operads i.e.
\begin{equation}
j(S\circ_{NAP,v}T)=J(S)\circ_{0,v}j(T),
\end{equation}
for any pair of trees $S,T$ and any vertex $v$ of $S$.

\section{Deformed NAP algebra}

\begin{defn}
Let $\lambda\in K$. We define the grafting operator
"$\leftarrow_{\lambda}"$, for any $S,T\in \mathcal{O}^{\lambda}$ by:
\begin{equation}
T \leftarrow_{\lambda} S=\big(\arbabst\circ_{v,\lambda}T\big)\circ_{w,\lambda}S.
\end{equation}
It is clear then that:
\begin{equation}
T \leftarrow_{\lambda} S=\sum_{v\in v(T)}{\lambda^{\left|S\right|h(v)}T\leftarrow_v S},
\end{equation}
where $T\leftarrow_v S$ is the tree obtained by grafting $S$ on the vertex $v$ of $T$, $h(v)$ is the height of $v$ in $T$ and $\left|S\right|$ is the weight of $S$.
\end{defn}
\begin{thm}\label{theoreme deforme}
For any $S,T,U \in \mathcal{O}^{\lambda}$, we have:
\begin{equation}
(U\leftarrow_{\lambda}T)\leftarrow_{\lambda} S-\lambda^{\left|S\right|}U\leftarrow_{\lambda}(T\leftarrow_{\lambda}S)=(U\leftarrow_{\lambda}S)\leftarrow_{\lambda}T-\lambda^{\left|T\right|}U\leftarrow_{\lambda}(S\leftarrow_{\lambda}T).
\end{equation}
\end{thm}
\begin{proof}
We have:
\begin{eqnarray*}
(U\leftarrow_{\lambda}T)\leftarrow_{\lambda}S&=&\sum_{v\in
  v(U\leftarrow_{\lambda}T)}{\sum_{w\in
    v(U)}{\lambda^{\left|S\right|h(v)+\left|T\right|h(w)}}~~(U\leftarrow_w T)\leftarrow_v S}\\
&=&\sum_{w\in v(U)}{\sum_{v\in v(T\rightarrow_w
    U)}{\lambda^{\left|S\right|h(v)+\left|T\right|h(w)}}~~(U\leftarrow_w
  T)\leftarrow_v S)}\\
&=&\sum_{w\in v(U)}{\sum_{v\in v(U) }{\lambda^{\left|S\right|h(v)+\left|T\right|h(w)}}~~(U\leftarrow_wT)\leftarrow_v S}\\
& &+\sum_{w\in v(U)}{\sum_{v\in v(T)
  }{\lambda^{\left|S\right|(h(v)+h(w)+1)+\left|T\right|h(w)}}~~U\leftarrow_w
  (T\leftarrow_v S)}.
\end{eqnarray*}
Similarly:
\begin{eqnarray*}
U\leftarrow_{\lambda}(T\leftarrow_{\lambda}S)&=&\sum_{v\in v(T)}{\sum_{w\in
    v(U)}{\lambda^{\left|S\right|h(v)+\left|T\leftarrow_v
        S\right|h(w)}}~~~U\leftarrow_w (T\leftarrow_v S)}\\
&=&\sum_{v\in v(T)}{\sum_{w\in
    v(U)}{\lambda^{\left |S\right |(h(v)+h(w))+\left |T\right |h(w)}}}~~~U\leftarrow_w (T\leftarrow_v S).
\end{eqnarray*}
Thus:
\begin{equation}\label{equationlambda}
(U\leftarrow_{\lambda}T)\leftarrow_{\lambda}S-\lambda^{\left|S\right|}U\leftarrow_{\lambda}(T\leftarrow_{\lambda}S)=\sum_{v,w\in
  v(U)}{\lambda^{\left|S\right|h(v)+\left|T\right|h(w)}(U\leftarrow_{w}T)\leftarrow_v S}.
\end{equation}
In Equation \eqref{equationlambda}, there is a symmetry in the terms $S$ and $T$ because $w,v$ are two vertices of $U$. As a result we prove Theorem \ref{theoreme deforme}.
\end{proof}

\subsection{Definitions for  multigraded operads} 

\begin{defn}
Let $\mathcal{P}$ and $\mathcal{Q}$ two multigraded operads. A
morphism of multigraded operads  is a morphism of operads  $a$ , wich moreover verifies:
\begin{equation}
a(n):\mathcal{P}_{n,k_1,\ldots,k_n}\longrightarrow\mathcal{Q}_{n,k_1,\ldots,k_n},
\end{equation}
for any $n\in\mathbb{N}^*$ and $k_1 ,\ldots,k_n \in\mathbb{N}^*$.

\end{defn}
\begin{defn}
Let $\mathcal{P}$ be a multigraded operad and $V=\bigoplus_{n>0}V_n$ be a graded vector space. We say that $V$ is a  graded $\mathcal{P}$-algebra, if there is a morphism of multigraded operads\\
$a:\mathcal{P}\rightarrow Endop(V).$ Let $(V,a)$ and $(W,b)$ be two  graded $\mathcal{P}$-algebras. Let
$\phi:(V,a)\rightarrow (W,b)$ be a linear homogenous map of degree
zero. $\phi$ is called morphism of graded $\mathcal{P}$-algebras if
for any $\mu\in\mathcal{P}(n),~~v_1,\ldots,v_n \in~V,$
\begin{equation*}
\phi\big(\mu(v_1,\ldots,v_n)\big)=\mu\big(\phi(v_1),\ldots,\phi(v_n)\big).
\end{equation*}
We define the   free  graded $\mathcal{P}$-algebra 
$\mathcal{F}_{\mathcal{P}}(V)$ on $V$ by the  following universal property: there exists $i:V\rightarrow \mathcal{F}_{\mathcal{P}}(V)$ such that
for any   graded $\mathcal{P}$-algebra $A$ and any linear map of degree zero $\varphi: V\rightarrow A$ there exists an
unique morphism of  graded $\mathcal{P}$-algebras
$\tilde{\varphi}:\mathcal{F}_{\mathcal{P}}(V)\rightarrow A$ such that
\begin{equation*}
\tilde{\varphi}\circ i=\varphi.
\end{equation*}
\end{defn}
We set $\mathcal{F}=\bigoplus_{n>0}\left((\prod_{k_1,\ldots,k_n
  >0}\mathcal{P}_{n,k_1,\ldots,k_n})\bigotimes V_{k_1}\otimes\ldots\otimes
V_{k_n}\right)/S_n$. By the same work as in section $1$, we show that
$\mathcal{F}$ is  the free graded $\mathcal{P}$-algebra over $V$.\\
{\bf{Special case:}} if for any $n\in\mathbb{N}^*,~~\dim V_n =1$ then:
\begin{eqnarray*}
\mathcal{F}_{\mathcal{P}}(V)&=&\bigoplus_{n>0}\prod_{k_1 ,\ldots,k_n
  >0}\mathcal{P}_{n,k_1,\ldots,k_n}/S_n\\
&=&\bigoplus_{n>0}\mathcal{P}_n /S_n
\end{eqnarray*}
\begin{rmk}
The deformed pre-Lie algebra $(\mathcal{F}_{\mathcal{O}^{\lambda}},\leftarrow_{\lambda} )$ is:
\begin{itemize}
\item  A right pre-Lie algebra if $\lambda=1$ \cite{MS}, \cite{AgraGam} \cite{S}.
\item  A right $NAP$ algebra if  $\lambda=0$ \cite{Liv}.
\end{itemize}
More precisely the morphisms of operads $i$ and $j$ defined above pass to the quotient by the action of the symmetric group. So we have a morphism of pre-Lie algebras\\
$\bar{i}:\mathcal{F}_{PL}\rightarrow\mathcal{F}_{\mathcal{O}^{1}}$ and a morphism of NAP algebras
$\bar{j}:\mathcal{F}_{NAP}\rightarrow\mathcal{F}_{\mathcal{O}^{0}}$. In addition 
$\bar{i}$  is the unique morphism of pre-Lie algebras such that
$\bar{i}(\racine)=\sum_{n>0}{\racinen}$. Similarly $\bar{j}$ is the unique morphism of NAP algebras such that $\bar{j}(\racine)=\sum_{n>0}{\racinen}$. 
\end{rmk}

\subsection{The multigraded operad $ \mathcal{PL}_{\lambda}$}

A pre-Lie algebra is an algebra over some binary quadratic operad. Similarly a deformed pre-Lie algebra in the above sense is a graded algebra over an multigraded operad  $\mathcal{PL}_{\lambda}$: Let $\mathcal{F}$ be the free multigraded operad spanned by the space
$\mathcal{F}(2)=\prod_{k,l>0}\mathcal{F}_{2,k,l}$. A basis of 
$\mathcal{F}(n)$ is a vector space endowed with a product and parentheses on $n$ variables indexed by $k_1 ,\ldots,k_n$. For example
a basis of $\mathcal{F}(2)$ is given by $(x_k y_l)$ and $(y_l x_k )$ for any $k,l>0$. A basis of $\mathcal{F}(3)$ is $\big( (x_k y_l )z_m\big)
,~~\big( x_k (y_l z_m)\big)$ and their permutations. Let $ r$ be the sub-module
of $\mathcal{F}(3)$ spanned by:
$$\big( (x_k y_l )z_m \big)-\lambda^{m}\big( x_k (y_l z_m)\big)-\big( (x_k
z_m)y_l \big) +\lambda^{l}\big( x_k (z_m y_l )\big).$$
Let $(R)$ be the ideal spanned by $r$. We denote
$\mathcal{PL}_{\lambda}=\mathcal{F}/(R)$. The following theorem
is the analogue of Theorem 1.9 of F. Chapoton and M. Livernet \cite{ChaLiv}.
\begin{thm}
The operad $\mathcal{PL}_{\lambda}$ is isomorphic to the multigraded operad
$\mathcal{O}^{\lambda}$ defined on the rooted trees.
\end{thm}
{\bf{Change of notations:}} in the following section, we denote by
``$\star_{\lambda}$''  the product ``$\leftarrow_{\lambda}$''
\begin{proof}
 We adapt the proof of Theorem 1.9 of \cite{ChaLiv}. We define a morphism of operad
$\phi:\mathcal{PL}_{\lambda}\rightarrow \mathcal{O}^{\lambda}$ and we will show that $\phi$  is an isomorphism. Since
$\mathcal{PL}_{\lambda}=\mathcal{F}/(R)$, it suffices to define $\phi$ on
$\mathcal{PL}_{\lambda}(2)=\mathcal{F}(2)$ in $\mathcal{O}^{\lambda}(2)$
and then extend $\phi$ on $\mathcal{F}$ by the universal property of free operad,  and to show that  $\phi$ vanishes on $(R)$. We define $\phi(x_k y_l)=\arbabxy$ and $\phi(y_l x_k
)=\arbabyx$. We check:

\begin{eqnarray*}
 &&\phi( \big( (x_k y_l )z_m \big)-\lambda^{m}\big( x_k (y_l z_m)\big)-\big( (x_k
z_m)y_l \big) +\lambda^{l}\big( x_k (z_m y_l )\big) )\\
&=&\phi(x_k
 y_l)\star_{\lambda}\phi(z_m)-\lambda^{m}\phi(x_k)\star_{\lambda}\phi(y_l z_m )-\phi(x_kz_m )\star_{\lambda}\phi(y_l)+\lambda^{l}\phi(x_k)\star_{\lambda}\phi(z_m y_l )\\
&=&0,
\end{eqnarray*}
hence $\phi(X)=0$ for any $X\in r$. The morphism $\phi$ hence vanishes on  $(R)$.\\
 Let $I$ a finite set formed by labels. We consider  maps:
$$\begin{array}{ccccl}
\underline{k}&:&I&\longrightarrow& \mathbb{N}^*\\
             & &a&\longmapsto&      k_a
\end{array}$$
Working with the species, consider $\phi_I :\mathcal{PL}_{\lambda}(I)\rightarrow \mathcal{O}^{\lambda}(I).$ seek
$\psi_I:\mathcal{O}^{\lambda}(I)\rightarrow \mathcal{PL}_{\lambda}(I)$ such that:
$$\left\{\begin{array}{ccc}\phi_I \circ \psi_I&=&id\\ \psi_I \circ
    \phi_I&=&id\end{array}\right.$$
We have  $\mathcal{PL}_{\lambda}(I)=\prod\mathcal{PL}_{\lambda}(I)_{\underline{k}}$
and
$\mathcal{O}^{\lambda}(I)=\prod\mathcal{O}^{\lambda}(I)_{\underline{k}}$. We look for
$\phi_{I,\underline{k}}:\mathcal{PL}_{\lambda}(I)_{\underline{k}}\rightarrow \mathcal{O}^{\lambda}(I)_{\underline{k}}$.
Show the result by induction on the cardinal of $I$.
\begin{itemize}
\item If $I=\{x\}$, for any $l\in\mathbb{N}^*$ we set  $\psi(\arbxl)=x_l$.
\item Let $n>0,$ we assume that the property is true for any finite set $I$ of cardinality less than or equal to  $n$. Let $I$ be a set of cardinality $n+1$. Let $T$ be a  rooted tree  labeled by $I$ (the weight of vertices are determined by $\underline{k}:I\rightarrow\mathbb{N}^*$).  Let $x$ be the index of the root of  $T$ and $l=\underline{k}(x)$ its weight. Modulo the permutation of the branches,
  $T$ is written in a unique way:
\begin{equation}
T=B[x,T_1 ,\ldots, T_P]=
\end{equation}
where for any $i\in\{1,\ldots,p\},~T_i$ is a tree labeled by a sub-set $J_i$ of $I$. We define the map $\psi_I$ by induction on the valence of the root:
\begin{itemize}
\item If $p=1$, then $T=B[x,T_1]=\arbxl\star_{\lambda} T_1$ hence $\psi_I
  (T)=\big( x_l \psi_I (T_1 )\big)$.
\item If $p>1$. We have:
\begin{eqnarray*}
T&=&B[x,T_1,\ldots,T_p]\\
 &=&B[x,T_2,\ldots,T_p]\star_{\lambda}T_1
 -\lambda^{\left|T_1\right|}\sum_{j=2}^{p}{B[x,T_2,\ldots,T_j \star_{\lambda}T_1,\ldots,T_p]}.
\end{eqnarray*}
Thus 
\begin{equation}
\psi_I (T)=\big( \psi_I ( B[x,T_2,\ldots,T_p])\psi_I
(T_1)\big)-\lambda^{\left| T_1\right|}\sum_{j=2}^{p}{\psi_I \big(
  B[x,T_2,\ldots,T_j \star_{\lambda}T_1 ,\ldots,T_p]\big)}.
\end{equation}
It remains to show that $\psi_I (T)$ does not depend on permutation of the branches $T_i$
for any $i\in\{1,\ldots,p\}$. We have:
\begin{eqnarray*}
T&=&B[x,T_2,\ldots,T_p]\star_{\lambda}T_1
 -\lambda^{\left|T_1\right|}\sum_{j=2}^{p}{B[x,T_2,\ldots,T_j
   \star_{\lambda}T_1,\ldots,T_p]}\\
&=&\left(B[x,T_3,\ldots,T_p]\star_{\lambda}T_2
  -\lambda^{\left|T_2\right|}\sum_{k=3}^{p}{B[x,T_3,\ldots,T_k
    \star_{\lambda}T_2 ,\ldots,T_p]} \right)\star_{\lambda}T_1\\
&&-\lambda^{\left|T_1\right|}B[x,T_2\star_{\lambda}T_1,T_3,\ldots,T_p]-\lambda^{\left|T_1\right|}\sum_{j=3}^{p}{B[x,T_2
  ,\ldots,T_j\star_{\lambda}T_1,\ldots,T_P]}\\
&=&\left(B[x,T_3,\ldots,T_p]\star_{\lambda}T_2 \right)\star_{\lambda }T_1
  -\lambda^{\left|T_2\right|}\sum_{k=3}^{p}{\left( B[x,T_3,\ldots,T_k
    \star_{\lambda}T_2 ,\ldots,T_p]\right)\star_{\lambda}T_1}\\
&&-\lambda^{\left|T_1\right|}\left(B[x,T_3,\ldots,T_p]\star_{\lambda}(T_2
  \star_{\lambda} T_1)+\lambda^{\left|T_2\star_{\lambda}
        T_1\right|}\sum_{j=3}^{p}{B[x,T_3,\ldots,T_j
    \star_{\lambda}(T_2\star_{\lambda}T_1),\ldots,T_p]}\right)\\
&&-\lambda^{\left|T_1\right|}\sum_{j=3}^{p}{B[x,T_3,\ldots,T_j
  \star_{\lambda}T_1,\ldots,T_p]\star_{\lambda}T_2}+\lambda^{\left|T_1\right|+\left|T_2\right|}\sum_{j=3}^{p}{B[x,t_3,\ldots,(T_j\star_{\lambda}T_1)\star_{\lambda}T_2,\ldots,T_p]}\\
&&+\lambda^{\left|T_1\right|+\left|T_2\right|}\sum_{j=3}^{p}{}\sum_{k=3,k\neq
  j}^{p}{B[x,T_3,\ldots,(T_j \star_{\lambda}T_1),\ldots,(T_k\star_{\lambda}T_2),\ldots,T_p]}.
\end{eqnarray*}
Let then:
\begin{equation}
A_{12}=\big( (\psi_I (B[x,T_3,\ldots,T_p])\psi_I (T_2 ))\psi_I
  (T_1) \big)-\lambda^{\left|T_1\right|}\big(\psi_I
  (B[x,T_3,\ldots,T_p])\big(\psi_I (T_2)\psi_I (T_1)\big)\big).
\end{equation}
Let $A_{21}$ the same term with  $T_1$ and $T_2$ permuted. We have
$A_{12}-A_{21}\in (R)$. It is clear that for any $j\in\{3,\ldots,p\}$,
\begin{eqnarray*}
B_{12}^{j}&=&\lambda^{\left|T_2\right|}\big(
\psi_I(B[x,T_3,\ldots,T_j\star_{\lambda}T_2,\ldots,T_p])\psi_I
(T_1)\big)+\lambda^{\left|T_1\right|}\big(\psi_I
(B[x,T_3,\ldots,T_j\star_{\lambda}T_1,\ldots,T_p])\psi(T_2)\big)\\
&=&B_{21}^{j} 
\end{eqnarray*}
Similarly  $\sum_{k,j=3,k\neq j}^{p}{C_{12}^{jk}}=\sum_{k,j=3,k\neq
  j}^{p}{C_{21}^{jk}}$, where:
\begin{equation}
C_{12}^{jk}=\psi_I (B[x,T_3,\ldots,(T_j\star_{\lambda}T_1),\ldots, (T_K \star_{\lambda}T_2),\ldots,T_p]).
\end{equation}
For any $j\in\{3,\ldots,p\}$ we have:
\begin{equation}
D_{12}^{j}=\lambda^{\left|T_1\right|+\left|T_2\right|}\psi_I
\big(B[x,T_3,\ldots,(T_j\star_{\lambda}T_1)\star_{\lambda}T_2,\ldots,T_p]\big)+\lambda^{\left|T_1\right|}\psi_I \big(B[x,T_3,\ldots,T_j\star_{\lambda}(T_2\star_{\lambda}T_1),\ldots,T_p]\big),
\end{equation}
coincides with $D_{21}^{j}$ by Theorem \ref{theoreme deforme}.
\end{itemize}
\end{itemize}
It was shown that $\psi_I (T)$  does not depend on permutations of the branches
$T_1$ and $T_2$. By induction  we can show that $\psi_I (T)$ does not depend on the order of the branches $T_1$ and $T_j$ for any $j\in\{2,\ldots, p\}$. Similarly we prove that $\psi_I (T)$ does not depend on the order of the branches
$T_j$ and $T_k$ for any $k,j\in\{1,\ldots,p\}$ with $j\neq k$. Hence the uniqueness of $\psi_I (T)$.
\end{proof}

\end{document}